\documentclass[a4paper, 11pt]{article}
\usepackage{amsfonts,latexsym,rawfonts,amsmath,amssymb,amsthm}

\usepackage{color}
\usepackage[applemac]{inputenc}
\usepackage[T1]{fontenc}
\usepackage{textcomp}
\usepackage{geometry}
\usepackage{graphicx}
\usepackage{mathptmx}
\usepackage{hyperref}

\numberwithin{equation}{section}

\newcommand{\pd}[2]{\frac {\partial #1}{\partial #2}}

\newcommand{\al}{\alpha}
\newcommand{\bb}{\beta}
\newcommand{\la}{\lambda}
\newcommand{\La}{\Lambda}
\newcommand{\oo}{\omega}
\newcommand{\Om}{\Omega}

\newcommand{\dd}{\delta}
\newcommand{\Na}{\nabla}

\newcommand{\ee}{\epsilon}
\newcommand{\si}{\sigma}

\newcommand{\te}{\theta}

\newcommand{\beq}{\begin{equation}}
\newcommand{\eeq}{\end{equation}}
\newcommand{\beqs}{\begin{eqnarray*}}
\newcommand{\eeqs}{\end{eqnarray*}}
\newcommand{\beqn}{\begin{eqnarray}}
\newcommand{\eeqn}{\end{eqnarray}}
\newcommand{\beqa}{\begin{array}}
\newcommand{\eeqa}{\end{array}}

\def\td{\tilde}
\def\p{\partial}
\def\b{\bar}

\def\PP{{\mathbb P}}
\def\CC{{\mathbb C}}

\def\ri{\rightarrow}

\def\un{\underline}

\def\si{\sigma}

\def\tr{{\rm tr}}

\def\i{{\sqrt{-1}}}

\def\Up{\Upsilon}

\newtheorem{prop}{Proposition}[section]
\newtheorem{theo}[prop]{Theorem}
\newtheorem{lem}[prop]{Lemma}

\newtheorem{cor}[prop]{Corollary}
\newtheorem{rem}[prop]{Remark}

\newtheorem{defi}[prop]{Definition}

\title{The Futaki invariant on the blowup of K\"ahler surfaces\footnote{
The first-named author is  supported in part by NSFC No. 11001080 and No. 11131007 and
the second-named author
by NSFC No. 11101206.
}}
\author{Haozhao Li,\quad Yalong Shi  }

\begin{document}
\bibliographystyle{plain}

\date{}

\maketitle

\begin{abstract}
  We prove the expansion formula for the classical
  Futaki invariants on the blowup of K\"ahler surfaces, which explains the balancing
  condition of Arezzo-Pacard in \cite{[AP2]}.  The relation with Stoppa's result
   \cite{[Stoppa]} is also discussed.
\end{abstract}

\tableofcontents

\section{Introduction}
In \cite{[Ca1]}, E. Calabi introduced the extremal K\"ahler metric
on a compact K\"ahler manifold, which is a critical point of the
Calabi functional. A special case of extremal K\"ahler metrics is
the constant scalar curvature K\"ahler (cscK for brevity) metrics. The uniqueness of extremal
K\"ahler metric was proved by Chen-Tian in \cite{[CT]}. However, the
existence of extremal K\"ahler metrics or constant scalar curvature
metrics is a long standing difficult problem, which is closely
related to some stabilities conditions in algebraic geometry. In
some special cases, the extremal metrics can be constructed
explicitly and they have many interesting properties (cf.
\cite{[Ca1]} \cite{[TF1]}\cite{[ACGF]}). In a series of papers,
Arezzo-Pacard \cite{[AP1]} \cite{[AP2]} and Arezzo-Pacard-Singer
\cite{[APS]}  get a general existence result on the blowup of a
K\"ahler manifold with extremal K\"ahler metrics or constant scalar
curvature metrics at finite many points with some conditions by
using a gluing method.

To state Arezzo-Pacard's theorem, we introduce some notations. Let
$(M, \oo)$ be a compact K\"ahler manifold with a K\"ahler metric
$\oo$, and $K$ the group of automorphisms of $M$ which are also
exact symplectorphisms of $(M, \oo)$. There is a normalized moment
map
$$\xi: M\ri \frak k^*,$$
where $\frak k$ is the Lie algebra of $K.$ Moreover, for any $X\in
\frak k$, $\langle \xi, X\rangle$ is a Hamiltonian generating
function on $X$ and satisfying the normalization condition \beq
\int_M\;\langle \xi, X\rangle\,\oo_g^n=0.\label{eq:xi}\eeq
\begin{theo}\label{theo:AP}(Arezzo-Pacard \cite{[AP2]})Let $(M, \oo)$ be a compact K\"ahler
manifold with constant scalar curvature metric $\oo$, and $\pi: \td
M\ri M$ the blow up at distinct points $\{p_1, \cdots, p_k\}$
satisfying the following conditions
\begin{enumerate}
  \item[(1)] $\xi(p_1), \cdots, \xi(p_n)$ span $\frak k;$
  \item[(2)] there exists $a_1, \cdots, a_k>0$ such that
  \beq \sum_{j=1}^k\,a_j \xi(p_j)=0\in \frak k.\label{eq0001}
  \eeq
\end{enumerate}Then, there exist $c>0, \ee_0>0$ and for all $\ee\in (0,
\ee_0)$, there exists on $\td M$ a constant scalar curvature metric
$\oo_{\ee}$ in the K\"ahler class
$$\oo_{\ee}\in \pi^*[\oo]-\ee^2(a_{1, \ee}^{\frac 1{m-1}}c_1([E_1])+\cdots+ a_{k, \ee}^{\frac 1{m-1}}c_1([E_k]))$$
where $a_{j, \ee}$ satisfies $|a_{j, \ee}-a_j|\leq c\ee^{\frac
2{2m+1}}.$

\end{theo}

Condition (2) in Theorem \ref{theo:AP} is called the balancing
condition and it should be related with the stability property of
the blowing up manifold. In \cite{[Stoppa]} J. Stoppa gives the
expansion of the Donaldson-Futaki invariant on the blown up manifold and he
shows that the conditions (2) is naturally related to the Chow
stability of 0-dimensional cycles. Using this formula, he
proved that if we blow up a cscK manifold with integral K\"ahler
class $[\oo]$ at a Chow unstable 0-cycle $\sum_i a_i^{n-1}p_i$,
then for any rational $0<\epsilon\ll1$, the class $\pi^*[\oo]-\epsilon(\sum_ia_iE_i)$ does not contain a cscK metric, since this new polarized manifold is K-unstable.

A natural question is whether we can remove the rationality
assumption in Stoppa's theorem. Recall that when the K\"ahler
class is polarized by an ample line bundle $L$ (hence $M$ is
 a projective algebraic manifold) and the holomorphic vector field $X$
 generates a $\CC^*$ action, the Donaldson-Futaki invariant for the
  induced product test configuration coincides with the classical Futaki
  invariant of $X$ up to a universal constant \cite{[Do1]}.
  Since the vanishing of Futaki invariant is an obstruction
  to the existence of cscK metric, we can prove the non-existence
  of cscK metrics by a corresponding expansion formula
  for the classical Futaki invariant on the blown up manifold.
  In this paper, we will prove such a formula. For technical reasons, we restrict our attention to complex dimension 2.

 \begin{theo}\label{theo:main1}Let $\pi: \td M\ri M$ be the blowing up map of a compact K\"ahler surface
$M$ at the distinct points $\{p_1, p_2, \cdots, p_n\}. $ If the holomorphic
vector field $X$ on $M$
vanishes and is non-degenerate at $p_i(1\leq i\leq n)$, then we have
$$f_{\td M}(\td \Om_{\ee}, \td X)=f_M(\Om,   X)+\sum_{i=1}^n\nu_{p_i}(\Om, X)\cdot \ee_i+O(\ee^2),$$
where $\td X$ is the natural holomorphic extension of $X$ over $\td M$, and the K\"ahler class $\td \Om_{\ee}$ is
$$\td
\Om_{\ee}=\pi^*\Om-\sum_{i=1}^n\ee_ic_1([E_i]).$$ Here $\ee_i>0$
are small numbers and $\nu_{p_i}(\Om, X)$ are given by
$$ \nu_{p_i}(\Om,
X)=   -2\tr_{\Om}(X)(p_i)+\frac {2}{3\Om^2}J_{M}(\Om, X)
=-2(\te_X-\underline{\te_X})(p_i),$$
where $\te_X$ is the holomorphy potential of $X$ with respect to $\Om$, and $\underline{\te_X}$ is the average of $\te_X$.

\end{theo}

The notations in Theorem \ref{theo:main1} will be introduced in
Section \ref{sec2}. Note that $(\te_X-\underline{\te_X})(p_i)$ is independent of the choices of $\oo_g$ and $\te_X$, see Lemma \ref{lem:potential}.
When the manifold is a projective algebraic surface, the polarization is asymptotically Chow stable and the holomorphic vector field generates
a $\CC^*$ action, then  $(\te_X-\underline{ \te_x})(p_i)$ equals the Chow weight of $p_i$ up to a universal constant factor and hence
our result coincides with Stoppa's. For details, see section \ref{secStoppa}.\\

The proof of Theorem \ref{eq0001} is based on
Futaki \cite{[Futaki]} and Tian's localization formula in
\cite{[Tian98]} for the Futaki invariant, which essentially uses
Bott's residue formula for characteristic numbers in \cite{[Bott1]}.
However,  Bott's residue formula needs the non-degeneracy condition
on the holomorphic vector fields, and it will be difficult to remove
this condition when calculating the Futaki invariant. When we
consider the blown up manifold as in Arezzo-Pacard's result, under the
non-degeneracy assumption the induced holomorphic vector field on
the blown up manifold may still be degenerate somewhere and we
need to calculate the residue carefully in this case.

The key step in the proof of Theorem \ref{theo:main1} is to
calculate the  residue formula in the degenerate case, and we use
only  elementary calculus. It should be generalized to higher
dimensions. We note that there is a vast amount of literatures
discussing various residue formulas on $\CC^n$ (cf. \cite{[PTY]} \cite{[Tsikh]}
and references therein), but few of them focus on the case in the
K\"ahler manifolds, which usually involves K\"ahler metrics. The
calculation in this paper might be the first
step toward this direction. \\

A direct corollary of Theorem \ref{theo:main1} is the following
result, which  gives a partial converse of Theorem \ref{theo:AP} in
the special case of K\"ahler surfaces:

\begin{cor}\label{cor:main} Let $\pi: \td M(p_1, \cdots, p_n)\ri M$ be the blowing up map of
a compact K\"ahler surface $M$ at the points $\{p_i\}\subset
Zero(X)$, where  $X\in \frak h_0(M)$ is {non-degenerate} at
$\{p_i\}$. If  \beq \sum_{i=1}^n\,\langle\xi, X \rangle(p_i)\cdot
\ee_i\neq0,\label{eq0002}\eeq where $\xi$ satisfies the
normalization condition (\ref{eq:xi}) and $\ee_i>0$ are small,  then $\td M$ has no constant
scalar curvature metrics in the K\"ahler class
$$\td
\Om_{\ee}=\pi^*\Om-\sum_{i=1}^n\,\ee_i \,c_1([E_i]).$$

\end{cor}
 In fact, it is well-known that the moment map $\xi$ under the
 normalization condition (\ref{eq:xi})
can be characterized by $\langle\xi,
X\rangle=\te_X-\underline{\te_X}$. Therefore, Corollary \ref{cor:main} follows directly from Theorem
\ref{theo:main1}.
Corollary \ref{cor:main} gives a criterion on the non-existence of
constant scalar curvature metrics on the blown up manifold.
Moreover, the condition (\ref{eq0002}) may be related to the $\bar
K$-stability, which is introduced by Donaldson in \cite{[Do2]}\cite{[Do3]}.\\

The structure of the paper is as follows: In section \ref{sec2}, we include basic facts concerning the Futaki invariant and also outline the proof
of the localization formula of Futaki and Tian. To state our result in a clear way, we also define some local invariants on the zero locus of
a holomorphic vector field $X$. In section \ref{sec3}, we prove Theorem \ref{theo:main1} in the case when the blow up center is an isolated zero point of $X$,
and in section \ref{sec4}, we consider the case when the blow up center lies on a 1-dimensional component of the zero locus of $X$.
Note that the extension of $X$ on the blow up manifold is degenerate if and only if the blow up center $p$ is an isolated zero point of $X$ and the
linearization of $X$ at $p$ is not semisimple. This is proved in section \ref{sec3.1}. The proof of the degenerate case is the most technical part
of our paper, and occupies section \ref{sec3.3} and \ref{sec3.4}. Then in section \ref{secEg}, we apply our result to the blowup of $\CC\PP^1\times
\CC\PP^1$ at 2 or 3 points. The Futaki invariants in the former case has already been calculated by LeBrun and Simanca in \cite{[LB1]}. Our method
can also obtain a full expression for the Futaki invariant. For simplicity, we only write down the first order terms, which suffices to prove the
non-existence of cscK metrics in some K\"ahler classes. Finally in section \ref{secStoppa}, we compare our result with that of
Stoppa.\\

In a forthcoming paper, we will use a different method to get the
expansion of the Futaki invariant on compact  K\"ahler manifolds
of higher dimensions and  general holomorphic vector fields.
As an application, we will show  Corollary \ref{cor:main} for
more general cases.

\section{Preliminaries}\label{sec2}
In this section, we will recall briefly the localization
formula of  Bott \cite{[Bott1]},  Futaki \cite{[Futaki]} and Tian
 \cite{[Tian98]} for the calculation of the Futaki
 invariant.
\def\zero{{\rm Zero}}

 Let $(M,\oo_g)$ be a compact K\"ahler manifold, where $\oo_g=g_{i\bar j}dz^i\wedge d\bar z^j.$
 Here we adopt the convention that $\oo$ and $Ric(\oo)$ are defined without the usual
  ``$\sqrt{-1}$" factor.
 Let $\frak h_0(M)$ be the space of holomorphic vector fields with nonempty zero locus. For any $X\in \frak
h_0(M)$, we denote by $\zero(X)$ the zero set of $X$, which consists
of complex subvarieties $Z_{\la}(\la\in \La)$. We say $X$ is
non-degenerate on $Z_{\la}$, if $Z_{\la}$ is smooth and $\det(DX|_{TM/TZ_{\la}})$ is nowhere
zero along $Z_{\la}.$ $X$ is called non-degenerate on $M$ if $X$ is
non-degenerate on all the $Z_{\la}(\la\in \La).$

The
holomorphy potential of $X\in \frak h_0(M)$ with respect to $\oo_g$, denoted by $\theta_X$, is given by the equation
$$i_X\oo_g=-\bar \p \theta_X.$$
Such a $\te_X$ always exists and is unique up
to a constant. Note that the function $\te_X$ restricted on any
$Z_{\la}$ is a constant, and we define
$$\tr_{\Om}(X)_{Z_{\la}}=\te_X|_{Z_{\la}},$$
where $\Om=[\frac {\i}{2\pi}\oo_g]$ is the K\"ahler
class of $\oo_g$. Note that $\tr_{\Om}(X)_{Z_{\la}}$ depends on the choice of $\oo_g$ and $\theta_X$. However, we have

\begin{lem}\label{lem:potential}
  Let $\underline{\theta_X}$ be the average of $\theta_X$, then the value of $\theta_X-\underline{\theta_X}$ on a zero point of $X$ is independent of the choices of $\oo_g$ in $\Om$ and $\theta_X$.
\end{lem}

\begin{proof}
  First we fix the K\"ahler form $\oo_g$, then $\te_X$ is unique up to adding a constant. Then obviously $\theta_X-\underline{\theta_X}$ is independent of the choice of $\te_X$. Now let's fix a $\oo_g$ and $\te_X$ with $\underline{\te_X}=0$. We change $\oo_g$ by $\oo_\phi=\oo_g+\p\b\p \phi$. Then we can choose the holomorphy potential with respect to $\oo_\phi$ to be $\te_X-X(\phi)$. Since $X(\phi)$ vanishes on any zero point of $X$, we need only to prove $\int_M (\te_X-X(\phi))\oo_\phi^n=0$.
   Let $f(t):=\int_M (\te_X-tX(\phi))\oo_{t\phi}^n$. Then $f(0)=0$ and
   $$f'(t)=n\int_M(\te_X-tX(\phi))\p\b\p \phi\wedge \oo_{t\phi}^{n-1}-\int_M X(\phi)\oo_{t\phi}^n.$$
   Observe that $0=i_X(\p\phi\wedge\oo_{t\phi}^n)=X(\phi)\oo_{t\phi}^n-nt\p\phi\wedge i_X\oo_{t\phi}\wedge \oo_{t\phi}^{n-1}$. Integrate this and using integration by parts, we get directly $f'(t)=0$, hence $f(1)=0$.
\end{proof}

The
Futaki invariant of  $\Om$ and the
holomorphic vector field $X$ is defined by \beq f_M(\Om,
X)=\Big(\frac {\i}{2\pi}\Big)^{n}\int_M\; X(h_g)\,\oo_g^n
,\label{eqfu}\eeq where $h$ is a function satisfying
$$s(g)-\un s=\Delta_gh_g.$$

Now we start with some general discussions.Let $\phi$ be
any symmetric $GL$-invariant polynomial of degree $n+1$, and $E$ a
vector bundle over $M$. Assume that $h$ is a hermitian metric on $E$
and $\te_X(h)$ be an $End(E)$-valued function satisfying \beq  \bar
\p \theta_X(h)=-i_XR(h)\label{eq100}\eeq where   $R(h)$ denotes the
curvature of $h.$ Then we can check that
$$i_X\phi(R(h)+\theta_X(h))=-\bar\p \phi(R(h)+\theta_X(h)).$$
 We define a
$(1, 0)$ form $\eta$ on $M\backslash \zero(X)$ by $\eta(Y)=g(Y, \bar
X)/g(X, \bar X)$ for any $Y\in \frak h_0(M)$ and we define a formal
series of forms by
$$\al=\phi(\theta_X(h)+R(h))\wedge \frac {\eta}{1+\bar\p \eta}.$$
Direct calculation shows that \beq \phi(\theta_X(h)+R(h))-\bar\p
\al-i_X\al=0.\label{eq006}\eeq Let $[\bb]_{k}$ denote the degree $k$
term  in  $\bb.$ By (\ref{eq006}) we have \beq
[\phi(\theta_X(h)+R(h))]_{2n}=\bar \p [\al]_{2n-1}.\label{eq007}\eeq
Let $B_{\ee}(Z_{\la})$ be an $\ee$-neighborhood of $Z_{\la}$. Using
(\ref{eq007}) and the Stokes formula, we have \beqn
\int_M\;\phi(\theta_X(h)+R(h))&=&\lim_{\ee\ri
0^+}\int_{M\backslash\cup_{\la}B_{\ee}(Z_{\la})}\,\phi(\theta_X(h)+R(h))\nonumber\\
&=&\lim_{\ee\ri
0^+}\int_{M\backslash\cup_{\la}B_{\ee}(Z_{\la})}\,\bar \p
[\al]_{2n-1} \nonumber
\\&=&-\sum_{\la\in
\La}\lim_{\ee\ri 0^+}\int_{\p
B_{\ee}(Z_{\la})}[\al]_{2n-1},\label{eq008} \eeqn where $\p
B_{\ee}(Z_{\la})$ has the induced orientation such that the last
equality holds. The following result was essentially  proved by Bott
in \cite{[Bott1]}, and the readers are referred to Theorem 5.2.8 of
\cite{[Futaki]} for the details.

\begin{lem}\label{lem:bott}(\cite{[Bott1]}\cite{[Futaki]}) If $X$ is non-degenerate on $M$, then
$$\lim_{\ee\ri 0}\,\int_{\p B_{\ee}(Z_{\la})}\;\phi(\theta_X(h)+R(h))
\cdot \frac {\eta}{1+\bar\p \eta}=-(-2\pi \i)^{\nu
}\int_{Z_{\la}}\,\frac
{\phi(\theta_X(h)+R(h))}{\det(L_{\la}(X)+K_{\la})},$$ where $\nu$ is
the complex codimension of $Z_{\la}$ in $M$, $L_{\la}(X)$ is the operator $L_{\la}(X)(Y)=(\Na_YX)^{\perp}$ for $Y\in N_{M|Z_{\la}}$, and $K_{\la}$ is the
curvature form of the induced metric on
$N_{M|Z_{\la}}$ by $g.$ \\
\end{lem}

Now we would like to apply Lemma \ref{lem:bott} to the calculation of Futaki
invariants. Direct calculation shows that the Futaki invariant can
be expressed by \beqn &&(n+1)2^{n+1}f_M(\Om,
X)\nonumber\\&=&\sum_{j=0}^n(-1)^j\frac 1{j!(n-j)!}\Big(\frac
{\i}{2\pi}\Big)^{n}\int_M\;\Big((-\Delta_g\te_X+Ric(g)+(n-2j)(\te_X+\oo_g)
)^{n+1}\nonumber\\&&- (\Delta_g\te_X-Ric(g)+(n-2j)(\te_X+\oo_g)
)^{n+1}\Big)- n2^{n+1}\mu\cdot\Big(\frac
{\i}{2\pi}\Big)^{n}\int_M\;(\te_X+\oo_g)^{n+1},\nonumber\\
\label{eqfutaki}\eeqn where $\mu=\frac {c_1(M)\cdot
\Om^{n-1}}{\Om^n}.$ In the following, we want to choose the
polynomial $\phi$ and the vector bundle $E\ri M$ such that
(\ref{eqfutaki}) can be simplified by Lemma \ref{lem:bott}. We assume without loss of
generality that $\Om=c_1(L)$ for a holomorphic line bundle $L\ri M$. If we
choose \beq \phi=\frac 1{(n+1)!}\tr(x_1x_2\cdots x_{n+1}), \quad
E_1^j=K_M^{-1}\otimes L^{n-2j}.\label{eq022}\eeq then we have
$$R(h)=Ric(g)+(n-2j)\oo_g,\quad \theta_X(h)=-\Delta \te_X+(n-2j)\te_X,$$
where $R(h)$ is the curvature of the Hermitian metric on $E_1^j$ and
$\te_X(h)$ is determined by (\ref{eq100}). Therefore, we have \beqn
&&
\int_M\;\Big(-\Delta_g\te_X+Ric(g)+(n-2j)(\te_X+\oo_g) \Big)^{n+1}\nonumber\\
&=&-\sum_{\la\in \La}\lim_{\ee\ri 0^+}\int_{\p
B_{\ee}(Z_{\la})}\Big(-\Delta_g\te_X+Ric(g)+(n-2j)(\te_X+\oo_g)
\Big)^{n+1}\wedge \sum_{k=0}^{n-1}(-1)^k\eta\wedge (\bar
\p\eta)^{k}.\nonumber\\\label{eq023} \eeqn Similarly, if we choose
$E_2^j=K_M\otimes L^{n-2j}$ in (\ref{eq022}), we have
$$R(h)=-Ric(g)+(n-2j)\oo_g,\quad \theta_X(h)=\Delta \te_X+(n-2j)\te_X.$$
Combining this with (\ref{eq008}), we have \beqn &&
\int_M\;\Big(\Delta_g\te_X-Ric(g)+(n-2j)(\te_X+\oo_g) \Big)^{n+1}\nonumber\\
&=&-\sum_{\la\in \La}\lim_{\ee\ri 0^+}\int_{\p
B_{\ee}(Z_{\la})}\Big(\Delta_g\te_X-Ric(g)+(n-2j)(\te_X+\oo_g)
\Big)^{n+1}\wedge \sum_{k=0}^{n-1}(-1)^k\eta\wedge (\bar
\p\eta)^{k}. \nonumber\\\label{eq024}\eeqn For the last term of
(\ref{eqfutaki}), we choose $E=L^{n+1-2k}$ and do the same
calculation as above, \beq \int_M\; (\te_X+\oo_g)^{n+1}
=-\sum_{\la\in \La}\lim_{\ee\ri 0^+}\int_{\p
B_{\ee}(Z_{\la})}(\te_X+\oo_g)^{n+1}\wedge
\sum_{k=0}^{n-1}(-1)^k\eta\wedge (\bar \p\eta)^{k}.
\label{eq025}\eeq Combining the equalities
(\ref{eqfutaki})-(\ref{eq025}), we have \beqs  f_M(\Om,
X)&=&\sum_{\la\in \La} \Big(  I_{Z_{\la}}(\Om, X)-\frac n{n+1}\mu
 J_{Z_{\la}}(\Om, X)\Big), \eeqs where $I_{Z_{\la}}(\Om, X)$ and
$J_{Z_{\la}}(\Om, X)$ are defined by \beqn &&   I_{Z_{\la}}(\Om,
X)\nonumber\\&=& \frac 1{(n+1)2^{n+1}}\sum_{j=0}^n(-1)^j\frac
1{j!(n-j)!}\Big(\frac {\i}{2\pi}\Big)^{n}\Big(- \lim_{\ee\ri
0^+}\int_{\p B_{\ee}(Z_{\la})}\Big(-\Delta_g\te_X\nonumber\\
&&+Ric(g)+(n-2j)(\te_X+\oo_g) \Big)^{n+1}\wedge
\sum_{k=0}^{n-1}(-1)^k\eta\wedge (\bar \p\eta)^{k}\nonumber\\
&&+\lim_{\ee\ri 0^+}\int_{\p
B_{\ee}(Z_{\la})}\Big(\Delta_g\te_X-Ric(g)+(n-2j)(\te_X+\oo_g)
\Big)^{n+1}\wedge \sum_{k=0}^{n-1}(-1)^k\eta\wedge (\bar
\p\eta)^{k}\Big)\nonumber \\ \label{eq009}\eeqn and \beq
J_{Z_{\la}}(\Om, X)=- \lim_{\ee\ri 0^+}\Big(\frac
{\i}{2\pi}\Big)^{n}\int_{\p
B_{\ee}(Z_{\la})}(\te_X+\oo_g)^{n+1}\wedge
\sum_{k=0}^{n-1}(-1)^k\eta\wedge (\bar \p\eta)^{k}.\label{eq013}\eeq
Note that using the identities
$$\sum_{j=0}^n\;(-1)^j\binom{n}{j}(n-2j)^k=\left\{
                                             \begin{array}{ll}
                                               0, & \hbox{if}\; k<n \;\hbox{or } \;k=n+1; \\
                                               2^nn!, & \hbox{if}\;
k=n,
                                             \end{array}
                                           \right.
$$  the equality (\ref{eq009}) can be simplified to \beq
  I_{Z_{\la}}(\Om, X)=- \lim_{\ee\ri 0}\Big(\frac
{\sqrt{-1}}{2\pi}\Big)^n\,\int_{\p
B_{\ee}(Z_{\la})}\,(-\Delta_g\te_X+Ric(g))(\te_X+\oo_g)^n\wedge
\sum_{k=0}^{n-1}(-1)^k\eta\wedge (\bar \p\eta)^{k}.
\label{eq016}\eeq\\

Simple calculation shows that $-\Delta_g\te_X(Z_\la)=\tr(L_{\la}(X))$. Applying Lemma \ref{lem:bott} to (\ref{eq013})-(\ref{eq016}), we
have the following result:

\begin{theo}\label{lem:futaki}(\cite{[Futaki]}, \cite{[Tian98]}) For a K\"ahler class $\Om$ and non-degenerate $X\in \frak
h_0(M)$, the Futaki invariant is given by
$$f_M(\Om, X)=\sum_{\la\in \La}\, \Big(I_{Z_{\la}}(\Om,
X)-\frac n{n+1}\mu J_{Z_{\la}}(\Om, X))\Big),$$ where  $\mu=\frac
{c_1(M)\cdot \Om^{n-1}}{\Om^n}$ and  \beqn I_{Z_{\la}}(\Om,
X)&=&\int_{Z_{\la}}\, \frac { (\tr(L_{\la}(X))+c_1(M)
)(\tr_{\Om}(X)_{Z_{\la}}+\Om)^n}{\det(L_{\la}(X)+\frac {\i}{2\pi}K_{\la})},\label{eq:I}\\
J_{Z_{\la}}(\Om, X)&=&\int_{Z_{\la}}\, \frac {(
\tr_{\Om}(X)+\Om)^{n+1}}{\det(L_{\la}(X)+\frac {\i}{2\pi}K_{\la})}.
\label{eq:J}\eeqn

\end{theo}
Theorem \ref{lem:futaki} was  first proved by Futaki in
\cite{[Futaki]} for the first Chern class and by  Tian in
\cite{[Tian98]} for a general K\"ahler class.  For simplicity,  we
introduce the following notations:
\begin{defi}\label{defilocal}  For a K\"ahler class $\Om$ and $X\in \frak h_0(M)$ with $Zero(X)=\cup_\la Z_\la$,
the local Futaki invariant of $(\Om, X)$ on $Z_{\la}$ is defined by
$$f_{Z_{\la}}(\Om, X)= I_{Z_{\la}}(\Om,
X)-\frac n{n+1}\mu J_{Z_{\la}}(\Om, X)  ,$$ where $I_{Z_{\la}}(\Om,
X)$ and $J_{Z_{\la}}(\Om, X)$ are given by (\ref{eq:I}) and
(\ref{eq:J}) respectively when $X$ is nondegenerate on $Z_\la$, and by (\ref{eq016}) and (\ref{eq013}) in the general case. Moreover, we define \beq J_{M}(\Om,
X)=\sum_{\la\in \La}\;J_{Z_{\la}}(\Om, X)=\Big(\frac {\i}{2\pi}\Big)^n\int_M (\theta_X+\oo_g)^{n+1}=(n+1)\int_M \theta_X (\frac {\i}{2\pi}\oo_g)^n.\label{eq:JM}
\eeq\\

\end{defi}

 When $M$ has complex
dimension $2$, we can simplify the formula in Theorem
\ref{lem:futaki} as follows. Write the set of indices $\La=\La_0\cup
\La_1$ where $\La_i$ consists of all $\la$ with
$\dim_{\CC}Z_{\la}=i(i=0, 1)$ and we set
$$A_{\la}:=\tr(L_{\la}(X)),\quad  B_{\la}:=\tr_{\Om}(X)_{Z_{\la}},
\quad C_{\la}=\det(L_{\la}(X)). $$

The following result is a direct
corollary of Lemma \ref{lem:futaki}:

\begin{cor}\label{corfutaki}(\cite{[Tian98]}, \cite{[Tianbook]})Suppose that $\dim_{\CC}M=2$. If
$X\in \frak h_0(M)$ is non-degenerate on the zero set $Z_{\la}$, we
have \beq f_{Z_{\la}}(\Om, X)= \frac {A_{\la}B_{\la}^2- \frac
{2\mu}{3}B_{\la}^{3}}{C_{\la}} \label{eq002}\eeq for the case
$\la\in \La_0$, and \beqn f_{Z_{\la}}(\Om, X)&=& (2B_{\la}-2\mu
B_{\la}^2A_{\la}^{-1})\Om([Z_{\la}])\nonumber\\&&+(\frac
{2\mu}{3}A_{\la}^{-2}B_{\la}^3)c_1(M)([Z_{\la}])+(A_{\la}^{-1}B_{\la}^2-\frac
{2\mu}{3}A_{\la}^{-2}B_{\la}^3)(2-2g(Z_{\la}))\label{eq004}\eeqn for
$\la\in \La_1$,  where $g(Z_{\la})$ denotes the genus of $Z_{\la}.$
\end{cor}

Corollary  \ref{corfutaki}  is given by \cite{[Tian98]} and the
details of the proof is given by \cite{[Tianbook]}.

\begin{rem}
  We will also use the expression of $I_{Z_\la}(\Om,X)$ and $J_{Z_\la}(\Om,X)$ when $\dim_{\CC}Z_\la=0,1$ in later sections. So we write them down here:
  \begin{itemize}
    \item When $\dim_{\CC}Z_\la=0$, we have
    $$I_{Z_\la}(\Om,X)=\frac{A_\la B_\la^2}{C_\la}, \quad J_{Z_\la}(\Om,X)=\frac{ B_\la^3}{C_\la}.$$
    \item When $\dim_{\CC}Z_\la=1$, we have
    \beqs
    I_{Z_\la}(\Om,X)&=&2B_\la \Om([Z_\la])+A_\la^{-1}B_\la^2(2-2g(Z_\la)),\\
    J_{Z_\la}(\Om,X)&=&3A_\la^{-1}B_\la^2 \Om([Z_\la])-A_\la^{-2}B_\la^3\Big( c_1(M)([Z_\la])+2g(Z_\la)-2 \Big).
    \eeqs
  \end{itemize}
\end{rem}

\section{Blowing up at isolated zeros}\label{sec3}
 In this section we will calculate the Futaki invariant of the
blow up $\pi: \td M\ri M$ of a K\"ahler surface $M$ at an isolated
zero point $p\in M$ of $X$ with the exceptional divisor
$\pi^{-1}(p)=E.$ In this case, $X$ can be naturally extended to a holomorphic vector
field $\td X$ on $\td M$. We would like to calculate the
Futaki invariant of $(\td \Om_{\ee}, \td X)$ on $\td M$ where $\td
\Om_{\ee}=\pi^*\Om-\ee c_1([E])$.\\

Now we study the zero set of $\td X$ on the blown up $\td M.$ The
zero set $\zero(\td X)=\cup_{\la\in \td \La}\td Z_{\la}$ of $\td X$
on $\td M$ can be divided into two types: one coincides with the
zero set of $X$ on $M$ and we denote the set of the indices by $
\La$. The other belongs to the exceptional divisor $E$ and we denote
the set of the indices by $ \Up.$ Thus, the indices of the zeros
sets $\td Z_{\la}$ has the decomposition $\td \La=\La\cup \Up$. Set
$$\td \mu=\frac {c_1(\td
M)\cdot \td \Om_{\ee}}{\td \Om_{\ee}^2},\quad \dd:=\td
\mu-\mu=-\frac 1{\Om^2}\ee+O(\ee^2).$$ With these notations, we have

\begin{lem}\label{lemblowup}
$$f_{\td M}(\td \Om_{\ee}, \td X)=f_{M}(  \Om ,   X)-\frac n{n+1}\dd J_{M}(  \Om ,   X)
+\sum_{\la\in \Up}f_{\td Z_{\la}}(\td \Om_{\ee}, \td X)+\frac
n{n+1}\dd J_p(  \Om ,   X)-f_p(  \Om ,   X),$$ where $J_{M}(  \Om ,
X)$ is defined by (\ref{eq:JM}).

\end{lem}
\begin{proof}The Futaki invariant of $(\td \Om_{\ee}, \td X)$ on $\td
M$ is given by
$$f_{\td M}(\td \Om_{\ee}, \td X)=\sum_{\la\in \La}f_{\td Z_{\la}}(\td \Om_{\ee}, \td X)
+\sum_{\la\in \Up}f_{\td Z_{\la}}(\td \Om_{\ee}, \td X).$$ Note that
for any $\la\in \La$ we have \beqs f_{\td Z_{\la}}(\td \Om_{\ee},
\td X)&=&I_{\td Z_{\la}}(\td \Om_{\ee}, \td X)-\frac n{n+1}\td \mu
J_{\td Z_{\la}}(\td \Om_{\ee}, \td X)\\
&=&I_{  Z_{\la}}(  \Om ,   X)-\frac n{n+1}(\mu+\dd) J_{ Z_{\la}}(
\Om,   X)\\
&=&f_{  Z_{\la}}(  \Om ,   X)-\frac n{n+1}\dd J_{ Z_{\la}}( \Om,
X),\eeqs where we used the fact that $I_{\td Z_{\la}}(\td \Om_{\ee},
\td X)=I_{  Z_{\la}}(  \Om ,   X)$ and $J_{\td Z_{\la}}(\td
\Om_{\ee}, \td X)=J_{  Z_{\la}}(  \Om ,   X)$ since $\td Z_{\la} $
and $E$ are disjoint for $\la\in \La$.  Note that
$$f_{M}(  \Om ,   X)=\sum_{\la\in \La}f_{  Z_{\la}}(  \Om ,   X)+f_p(  \Om ,   X).$$
The lemma follows from the above equalities.
\end{proof}

\subsection{The zero set of the holomorphic vector field $\td X$}\label{sec3.1}
In this subsection, we will calculate the zero locus of the holomorphic
vector field $\td X$ on  $\td M$.

   Let $p\in M$ be an isolated zero
point of $X$ and $U$ be a neighborhood of $p$ with coordinates $(z,
w)$. Near the point $p$ the vector field $X$ can be written as \beq
X=X^1(z, w)\pd {}{z}+X^2(z, w)\pd {}{w},\label{eqX}\eeq where
$X^1(z, w)$ and $ X^2(z, w)$ are holomorphic functions on $U$.
We assume that  the functions $X^1$ and $X^2$ can be expanded on
$U$ near $p\in M$ as \beqn X^1(z, w)&=&a_1z+b_1w+\sum_{i+j\geq
2}c_{ij}z^iw^j,\label{eqX1}\\X^2(z, w)&=&a_2z+b_2w+\sum_{i+j\geq
2}d_{ij}z^iw^j,\label{eqX2}
 \eeqn where $a_i, b_i, c_{ij}$ and $d_{ij}$ are constants.  By our non-degenerate assumption, the
 matrix \[ \left( \begin{array}{ccc}
a_1 & b_1  \\
a_2 & b_2  \end{array} \right)\] is non-singular.

Consider the blowing up map $\pi: \td M\ri M$ at the point $p.$

\begin{lem}\label{lemzero}Let $p$ be a non-degenerate isolated zero point of $X$, where
$X$ is locally given by (\ref{eqX1})-(\ref{eqX2}). Then $\td X$ is non-degenerate if and only if
the matrix \[ \left( \begin{array}{ccc}
a_1 & b_1  \\
a_2 & b_2  \end{array} \right)\] is semisimple(i.e. diagonalizable.).

\end{lem}

\begin{proof}By a linear transform of coordinates, we may assume that the  matrix \[ \left( \begin{array}{ccc}
a_1 & b_1  \\
a_2 & b_2  \end{array} \right)\] is a Jordan canonical form. In particular, $a_2=0$.
 We choose
the coordinates on $\td U:=\pi^{-1}(U)$ as
$$\td U:=\Big\{((z, w), [\xi, \eta])\;\Big|\;(z, w)\in U, z\eta=w\xi\Big\}\subset U\times
\CC\PP^1,$$ which can be covered by two open sets $\td U_1=\{\xi\neq
0\}$ and $\td U_2=\{\eta\neq 0\}.$
 We choose the coordinate functions
$(u_1, v_1)$ on $\td U_1$ where $u_1=z, v_1=\frac {\eta}{\xi}$ and
the coordinate functions $(u_2, v_2)$ on $\td U_2$ where $u_2=\frac
{\xi}{ \eta}, v_2=w$. We want to compute the zero set of $\td X$
on $\td U_1$ and $\td U_2.$\\

On $\td U_1$ the  holomorphic
vector field $\td X$ can be written as
$$\td X=X^1(u_1, u_1v_1)\pd {}{u_1}+\frac 1{u_1}\Big(X^2(u_1, u_1v_1)-X^1(u_1, u_1v_1)v_1\Big)\pd {}{v_1}.$$
Thus, using the coordinates on $\td U_1$ the vector field $\td X$
can be expressed by
$$\td X=\td X^1(u_1, v_1)\pd {}{u_1}+\td X^2(u_1, v_1)\pd {}{v_1},$$
where
 \beqn \td X^1&=&u_1\Big(a_1+b_1v_1+\sum_{i+j\geq
2}\,c_{ij}u_1^{i+j-1}v_1^j\Big), \label{eqXonU11}\\
\td X^2&=&  (b_2-a_1)v_1-b_1v_1^2+\sum_{i+j\geq
2}\,(d_{ij}v_1^j-c_{ij}v_1^{j+1})u_1^{i+j-1}.\label{eqXonU12} \eeqn
Since $p$ is an isolated zero of $X$, the zero set of $\td X$ on
$\td U_1$ lies in the exceptional divisor and  it is given by
$$Z_1=\Big\{(u_1, v_1)\in \td U_1\;\Big|\;u_1=0,   (b_2-a_1)v_1-b_1v_1^2=0\Big\}\subset E\cap \td U_1. $$
which consists of the following cases:

\begin{itemize}
  \item If \[ \left( \begin{array}{ccc}
a_1 & b_1  \\
a_2 & b_2  \end{array} \right) =
\left( \begin{array}{ccc}
a & 0  \\
0 & a  \end{array} \right),
\] where $a\neq 0$,then $Z_1=E\cap \td U_1;$

  \item If \[ \left( \begin{array}{ccc}
a_1 & b_1  \\
a_2 & b_2  \end{array} \right) =
\left( \begin{array}{ccc}
a & 0  \\
0 & b  \end{array} \right),
\] where $a\neq b$, then $Z_1=\{p_1\}$ where $p_1$ has the coordinates
  $$p_1:(u_1, v_1)=(0, 0);
  $$
\item If \[ \left( \begin{array}{ccc}
a_1 & b_1  \\
a_2 & b_2  \end{array} \right) =
\left( \begin{array}{ccc}
a & 1  \\
0 & a  \end{array} \right),
\] where $a\neq 0$, then
  $Z_1=\{p_1\}$ where $p_1$ has the coordinates $$
  p_1:(u_1, v_1)=(0,0),$$
  and $\td X$ is degenerate at this point.
\end{itemize}

Now we calculate the zero set of $\td X$ on $\td U_2.$ Using the
coordinates on $\td U_2$ the holomorphic vector field $\td X$ can be
written as
$$\td X=\hat X^1(u_2, v_2)\pd {}{u_2}+\hat X^2(u_2, v_2)\pd {}{v_2},$$
where $\hat X^1(u_2, v_2)$ and $\hat X^2(u_2, v_2)$ are given by
\beqn \hat X^1(u_2, v_2)&=&b_1+(a_1-b_2)u_2+\sum_{i+j\geq 2}
(c_{ij}u_2^i-d_{ij}u_2^{i+1})v_2^{i+j-1},\label{eqXonU21}\\\hat
X^2(u_2, v_2)&=&v_2\Big(b_2+\sum_{i+j\geq
2}\,d_{ij}u_2^iv_2^{i+j-1}\Big). \label{eqXonU22} \eeqn Thus, the
zero set $Z_2$ of $\td X$ on $E\cap \td U_2$ is given by
$$Z_2=\Big\{(u_2, v_2)\;\Big|\;v_2=0,\;b_1+(a_1-b_2)u_2=0\Big\}.
$$ So we have:
\begin{itemize}
  \item If \[ \left( \begin{array}{ccc}
a_1 & b_1  \\
a_2 & b_2  \end{array} \right) =
\left( \begin{array}{ccc}
a & 0  \\
0 & a  \end{array} \right),
\] where $a\neq 0$, then $Z_2=E\cap \td U_2;$

  \item If \[ \left( \begin{array}{ccc}
a_1 & b_1  \\
a_2 & b_2  \end{array} \right) =
\left( \begin{array}{ccc}
a & 0  \\
0 & b  \end{array} \right),
\] where $a\neq b$, then $Z_2=\{q_1\}$ where $q_1$ has the coordinates
  $$q_1:(u_2, v_2)=(0,0).$$

  \item If \[ \left( \begin{array}{ccc}
a_1 & b_1  \\
a_2 & b_2  \end{array} \right) =
\left( \begin{array}{ccc}
a & 1  \\
0 & a  \end{array} \right),
\] where $a\neq 0$, then
  $Z_2=\emptyset$ .\\

\end{itemize}
Hence the lemma is proved.
\end{proof}

\subsection{The non-degenerate cases}

In this subsection, we will calculate the Futaki invariant of $(\td
\Om_{\ee}, \td X)$ for the non-degenerate cases in Lemma
\ref{lemzero}. It follows from Lemma \ref{lemblowup} that we need to
calculate the local Futaki invariant on the zero set of $\td X$
which lies in the exceptional divisor $E$ . The calculation is not
difficult since we have the nice formula in Corollary
\ref{corfutaki} when $\td X$ is non-degenerate.

\begin{theo}\label{theo001}If $\td X$ is non-degenerate on $\td M$, then
$$f_{\td M}(\td \Om_{\ee}, \td X)=f_M(\Om,   X)+\nu_p(\Om, X)\cdot \ee+O(\ee^2),$$
where $\nu_p(\Om, X)$ is given by
$$\nu_p(\Om, X)=-2\tr_{\Om}(X)(p)+\frac
{2}{3\Om^2}J_{M}(\Om, X).$$
\end{theo}

\begin{proof}
By Lemma \ref{lemblowup}, it suffices to compute $f_{\td
Z_{\la}}(\td \Om_{\ee}, \td X)(\la\in \Up)$ for the  non-degenerate
cases in Lemma \ref{lemzero}. We divide the proof into two cases:\\

\noindent{\bf Case 1:} For $a_1=b_2,  a_2=b_1=0$   the zero set of $\td X $ on $\td U$
is given by  $Z={E}.$ Note that by (\ref{eqXonU11})-(\ref{eqXonU12})
we have \beq \td A_E:=\tr(L_{E}(\td X))=a_1=\td C_E,\quad \td
\Om_{\ee}([E])=\pi^*\Om([E])-\ee E\cdot E=\ee. \label{eq301}\eeq To
calculate $B_E:=\tr_{\td \Om_{\ee}}(\td X)|_{E}$, we need to choose a suitable
K\"ahler metric on $\td M$ in the class $\td \Om_{\ee}$. We shall choose such a metric as Griffiths and Harris did in their
book \cite{[GH]}. The construction is as follows:

 Following the notations in
Section \ref{sec3.1}, the set $\td U_1$ has local coordinates $(u_1,
v_1)$ with $u_1=z,\;v_1=\frac {\eta}{\xi}$ and the exceptional
divisor is given by $u_1=0.$ The line bundle $[E]$ over $\td U$ has
transition function $z/w$ on $\td U_1\cap \td U_2$ and we can choose
a global section $\si$ of $[E]$ over $\td M$  by $\si=u_1$ on $\td
U_1$ and $\si=1$ on $\td M\backslash\td B_{\frac 12}$ where $\td
B_{r}:=\pi^{-1}(B_{\frac 12}(p)).$ Here $B_r(p)$ denotes the ball on
$M$ centered at $p$ with $|z|^2+|w|^2<r$ and we assume that $\td
B_1\subset \td U.$ Define the Hermitian metric $h_1$ of $[E]$ over
$\td U$ given in $\td U_1$ by
$$h_1=\frac {|\xi|^2+|\eta|^2}{|\xi|^2},$$
and $h_2$  the Hermitian metric of $[E]$ over $\td M\backslash E$
with $|\si|_{h_2}^2=1. $ Let $\rho_1, \rho_2$ be a partition of
unity for the cover $(\td B_{1}, \td M\backslash \td B_{\frac 12})$
of $\td M$ and let $h$ be the global Hermitian metric defined by
$$h=\rho_1h_1+\rho_2h_2.$$Then the function $|\si|_h^2$ on $\td B_{\frac 12}$ is given by \beq
|\si|_h^2=\frac {|\xi|^2+|\eta|^2}{|\xi|^2}\cdot
|u_1|^2=|u_1|^2+|u_1v_1|^2. \label{eq101}\eeq Given a K\"ahler
metric $\oo_g$ with the K\"ahler class $\Om=\frac {\i}{2\pi}[\oo_g]$
on $M$, the induced metric $\td \oo_{\ee}$ in the K\"ahler class
$\td \Om_{\ee}=\pi^*\Om-\ee c_1(M)$ is given by $$\td
\oo_{\ee}=\pi^*\oo_g+\ee \p\bar \p \log h.$$
Thus, the
holomorphy potential $\td \te_{\td X}$ of $\td X$ with respect to
$\td \oo_{\ee}$ is given by
$$\td \te_{\td X}=\pi^*\te_X-\ee\cdot \td X\left(\log |\si|^2_h\right).$$
Using the expression (\ref{eqXonU11})-(\ref{eqXonU12}) and
(\ref{eq101}), we have $\td X\left(\log |\si|^2_h\right)|_E=a_1.$
In conclusion, we have
 \beq \td B_E=\te_p-a_1\ee, \label{eq300}\eeq
where $\te_p=\te_X(p).$

Note that the genus of the exceptional divisor is zero  and $\td
\mu=\mu+\dd,$
 by Corollary \ref{corfutaki} we have \beqs f_{E}(\td \Om_{\ee}, \td
X)&=&(2
\td B_{E}-2\td \mu B_{E}^2\td A_{E}^{-1})\td \Om_{\ee}([E])\nonumber\\
&&+(\frac {2\td \mu}{3}\td A_{E}^{-2}\td B_{E}^3)c_1(\td M)([E])+(\td
A_{E}^{-1}\td B_{E}^2-\frac {2\td \mu}{3}\td A_{E}^{-2}\td
B_{E}^3)(2-2g(E))\\
&=&\frac {2\te_p^2}{a_1}-\frac {2\te_p^3}{3a_1^2}\mu-2\te_p\ee-\frac
{2\te_p^3}{3a_1^2}\dd+O(\ee^2), \eeqs where we used (\ref{eq301}).
On the other hand, using Corollary \ref{corfutaki} again we can
compute $f_p(\Om, X)$ and $J_p(\Om, X)$ as follows:
$$f_{p}(\Om, X)=\frac {2\te_p^2}{a_1}-\frac {2\te_p^3}{3a_1^2}\mu,
\quad J_p(\Om, X)=\frac { \te_p^3}{ a_1^2},$$ where we used the fact
that $A_p=2a_1, \; B_p=\te_p$ and $C_p=a_1^2.$
 Combining these with
Lemma \ref{lemblowup} we have $$ f_{\td M}(\td \Om_{\ee}, \td
X)=f_{M}(  \Om ,   X)-\frac 2{3}\dd J_{M}(  \Om ,   X)
-2\te_p\ee+O(\ee^2). $$\\

\noindent{\bf Case 2:} For $a_1\neq b_2$ and $a_2=b_1=0$, the zero set  $Z=\{p_1, q_1\}$ where
    $p\in \td U_1$ and $q\in \td U_2$ and the coordinates are given by
     $$p_1:(u_1, v_1)=(0, 0),\quad
     q_1:(u_2, v_2)=(0,0).$$
By the expression
(\ref{eqXonU11})-(\ref{eqXonU12}) of $\td X$ near $p_1$ we have
$$\td A_{p_1}=b_2,\quad \td B_{p_1}=\te_p-a_1\ee,\quad \td C_{p_1}=a_1(b_2-a_1),$$
where $\td B_{p_1}$ can be calculated as Case 1.
 Thus, the local Futaki invariant of $p_1$ is give by \beq
f_{p_1}(\td \Om_{\ee}, \td X)=\frac
{b_2(\te_p-a_1\ee)^2}{a_1(b_2-a_1)} - \frac
{2(\te_p-a_1\ee)^3(\mu+\dd)}{3a_1(b_2-a_1)}.\label{eq003}\eeq
Similarly,  by the expression (\ref{eqXonU21})-(\ref{eqXonU22}) of
$\td X$ near $q_1$ we have
$$\td A_{q_1}=a_1,\quad \td B_{q_1}=\te_p-b_2\ee,\quad \td C_{p_1}=b_2(a_1-b_2).$$
The local Futaki invariant of $q_1$ is give by \beq f_{q_1}(\td
\Om_{\ee}, \td X)=\frac {a_1(\te_p-b_2\ee)^2}{b_2(a_1-b_2)} - \frac
{2(\te_p-b_2\ee)^3(\mu+\dd)}{3b_2(a_1-b_2)}.\label{eq003}\eeq Next,
we calculate the local Futaki invariant of $p.$ Clearly, on the
point $p\in M$,
$$A_p=a_1+b_2,\quad B_p=\te_p,\quad C_p=a_1b_2$$
and we have
$$f_p(\Om, X)=\frac {(a_1+b_2)\te_p^2}{a_1b_2}-\frac {2\te_p^3\mu}{3a_1b_2},
\quad J_p(\Om, X)=\frac {\te_p^3}{a_1b_2}.$$ Collecting the above
results, we have \beqs f_{\td M}(\td \Om_{\ee}, \td X)&=&f_{M}(  \Om
, X)-\frac 2{3}\dd J_{M}(  \Om ,   X) +f_{p_1}(\td \Om_{\ee}, \td
X)+f_{q_1}(\td \Om_{\ee}, \td X)+\frac 23\dd J_p(\Om, X)-f_p(\Om,
X)\\
&=&f_{M}(  \Om ,   X)-\frac 2{3}\dd J_{M}(  \Om ,   X)
-2\te_p\ee+O(\ee^2). \eeqs
 The theorem is proved.

\end{proof}

\subsection{The degenerate case}\label{sec3.3}

In this subsection, we will calculate the Futaki invariant when $\td
X$ is degenerate on the exceptional divisor $E.$ In this case, the
calculation of Bott, Futaki and Tian fails and it should be related
to the general theory of Residue currents (cf. \cite{[Tsikh]} and
reference therein). However, when $M$ has complex dimension $2$ , we
can do the direct calculation using only the elementary calculus:

\begin{theo}\label{theo:010}Let $p$ be an isolated  zero of $X$. If $\td X$ is degenerate at
a zero point $\td p\in E,$ then the Futaki invariant of $(\td
\Om_{\ee}, \td X)$ is given by
$$f_{\td M}(\td \Om_{\ee}, \td X)=f_{M}(\Om, X)+\nu_p(\Om, X)\cdot \ee+O(\ee^2),$$
where
$$\nu_p(\Om, X)=-2\tr_{\Om}(X)(p)+\frac 2{3\Om^2}J_M(\Om, X).$$

\end{theo}
\begin{proof}
First, we claim that we can find a holomorphic coordinate transform around $p$ such that in the new coordinates, our holomorphic vector field contains only linear terms. The reason is the following:

We call a vector $\lambda=(\lambda_1,\dots,\lambda_n)\in \CC^n$ to be ``resonant'', if there is an integral relation of the form $\lambda_k=\sum_{i=1}^n m_i \lambda_i$, where $m_i$ are non-negative integers with $\sum_i m_i\geq 2$. And we say $\lambda$ belongs to the Poincar\'e domain if the convex hull of $\lambda_1,\dots,\lambda_n$ in $\CC$ does not contain the origin.

\begin{theo}[Poincar\'e, \cite{[Ar]},P190]
  If the eigenvalues of the linear part of a holomorphic vector field at a singular point (i.e. zero point) belong to the Poincar\'e domain and are non-resonant, then the vector field is biholomorphically equivalent to its linear part in a neighborhood of the singular point.
\end{theo}
 The idea of this theorem is that if the linear part of the vector field satisfies the ``non-resonant condition'', then we can construct a family of holomorphic coordinate transforms that eliminate the $k$-th order terms recursively for any $k\geq 2$. And if the eigenvalues are in the Poincar\'e domain, then the compositions of the coordinate transforms also converge to a holomorphic coordinate transform. The interested reader can find a detail discussion in \cite{[Ar]}.

 In our case, the linear part of $X$ clearly satisfies the conditions in Poincar\'e's theorem, so in the following discussion, we can assume without loss of generality that
$$X^1(z, w)=az+w,\quad X^2(z,w)=aw.$$
Then in the coordinates $(u_1,v_1)$ of previous subsections, $\td X$
can be written as the following on $\td U_1:$
$$\td X= u_1(a+v_1)\pd {}{u_1}-v_1^2\pd {}{v_1}.$$

It is clear from the discussion in Section 2 that in defining  $I_p$ and $ J_p$, we can use any family of domains shrinking to $p$. So in this section, we choose special domains to simplify the computation. Let $B_r$ be a sufficiently small ``distorted'' ball around $p_1$, defined by $|\td X|_g^2(u_1,v_1)\leq r^4$. We have the following
\begin{lem}\label{lem:main}
  Let $\phi$ be any smooth function on $\td U_1$. Then we have
$$\lim_{r\to 0+}\int_{\p B_r }\phi\eta\wedge\b\p\eta=\frac{4\pi^2}{a}\pd {\phi}{v_1}(0)-\frac{4\pi^2}{a^2}\phi(0).$$
\end{lem}

We use this lemma to calculate $f_{p_1}(\Om_{\ee}, \td X)$. First, note that for any smooth 2-form $\chi$, we have
$$\lim_{r\to 0+}\int_{\p B_r}\chi\wedge\eta=0.$$ This can be seen from the expression of $\eta$ in the next subsection. By
(\ref{eq013}), we have \beqs J_{p_1}(\Om_{\ee}, \td
X)&=&-\lim_{r\ri 0^+}\frac 1{4\pi^2}\int_{\p B_r}(\td
\te_{\td X}+\td \oo_g)^{3}\wedge \eta\wedge \bar
\p\eta\\&=&-\lim_{r\ri 0^+}\frac 1{4\pi^2}\int_{\p
B_r}\td \te_{\td X}^{3}\cdot \eta\wedge \bar \p\eta. \eeqs
To calculate the last term, we need to expand the function
$\td\te_{\td X}. $ In fact, near $p_1$ we
have
 \beqs \td X(\log|\si|^2_h)&=&u_1(a+v_1)\pd
{}{u_1}(\log(|u_1|^2+|u_1v_1|^2))-v_1^2\pd
{}{v_1}(\log(|u_1|^2+|u_1v_1|^2))\\
 &=&a+\frac{v_1}{1+|v_1|^2}.
  \eeqs It follows
that \beq \td \te_{\td X}(0)=\te_p-a\ee, \quad  \pd{\td\te_{\td X}}{v_1}(0)=-\ee.\label{eq501}\eeq
By Lemma
\ref{lem:main} and (\ref{eq501}), we have \beq J_{p_1}(\td \Om_{\ee},
\td X)=\frac {3(\te_p-a\ee)^2\ee}{a}+\frac
{(\te_p-a\ee)^3}{a^2} =\frac
{\te_p^3}{a^2}+O(\ee^2).\label{eq500}\eeq

Next, we calculate
$I_{p_1}(\td \Om_{\ee}, \td X).$ When $n=2$ we have
 \beqs
I_{p_1}(\td \Om_{\ee}, \td X)&=&-\lim_{r\ri 0}\Big(\frac
{\sqrt{-1}}{2\pi}\Big)^n\,\int_{\p B_r}\,(-\Delta_{\td
g}\td \te_{\td X}+Ric(\td g))(\td \te_{\td X}+\oo_{\td g})^n\wedge
\sum_{k=0}^{n-1}(-1)^k\eta\wedge (\bar \p\eta)^{k}\\
&=&-\frac 1{4\pi^2}\lim_{r\ri 0} \int_{\p B_r}\,-\Delta_{\td
g}\td \te_{\td X}\td \te_{\td X}^2\cdot\eta\wedge \bar\p \eta.  \eeqs
Direct computation shows that
 \beq -\Delta_{\td
g}\td \te_{\td X}(0)= a, \quad \pd{}{v_1}(-\Delta_{\td
g}\td \te_{\td X})(0)=-1. \label{eq502}
\eeq
 Combining this with (\ref{eq501}) and Lemma
\ref{lem:main}, we have
$$I_{p_1}(\td \Om_{\ee}, \td X)=\frac {2\te_p^2}{a}-2\te_p\ee+O(\ee^2).$$
On the other hand, we have \beq f_{p}(\Om, X)=\frac
{2\te_p^2}{a}-\frac {2\te_p^3}{3a^2}\mu, \quad J_p(\Om, X)=\frac
{ \te_p^3}{ a^2}.\label{eq021}\eeq Combining the above results, we
have \beqs f_{\td M}(\td \Om_{\ee}, \td X)&=&f_{M}(  \Om , X)-\frac
23\dd J_M(\Om, X)+f_{p_1}(\td \Om_{ \ee}, \td X)-f_{p}(\Om, X)+\frac
23\dd J_p(\Om, X)\\
&=&f_{M}(  \Om ,   X)-2\te_p\ee-\frac 23\dd J_M(\Om,
X)+O(\ee^2).\eeqs The theorem is proved.

\end{proof}

\subsection{Proof of Lemma \ref{lem:main}}\label{sec3.4}

We first write $\eta$ as (for simplicity, we sometimes use $(z^1,z^2)$ to denote $(u_1,v_1)$)
$$\eta=\eta_idz^i,$$
where $\eta_i=\frac{\alpha_i}{|\td X|^2_g}$, and $\alpha_i=g_{i\bar j}\overline{\td X^j}$. Direct computation shows that
$$|\td X|^2_g=g_{1\bar 1}|u_1(a+v_1)|^2-2 Re(g_{1\bar 2}u_1(a+v_1)\overline{v_1}^2)+g_{2\bar 2}|v_1|^4,$$
and
$$\eta\wedge\b\p \eta=\frac{(\alpha_i dz^i)\wedge \b\p \alpha_j \wedge dz^j}{|\td X|^4_g}=\frac{\alpha_i\alpha_{j,\b k}dz^i\wedge d\b z^k\wedge dz^j}{|\td X|^4_g},$$
where $\alpha_{\dots,\b k}$ means derivative in the direction of $\b z^k$. In our 2-dimensional case, we have
\beqs\eta\wedge\b\p \eta&=&\frac{\alpha_1\alpha_{2,\b 1}du_1\wedge d\b u_1 \wedge dv_1+\alpha_1\alpha_{2,\b 2}du_1\wedge d\b v_1 \wedge dv_1 }{|\td X|^4_g}+\\
& & \frac{\alpha_2\alpha_{1,\b 1}dv_1\wedge d\b u_1 \wedge du_1+\alpha_2\alpha_{1,\b 2}dv_1\wedge d\b v_1 \wedge du_1 }{|\td X|^4_g}\\
&=& \frac{(\alpha_1\alpha_{2,\b 1}-\alpha_2\alpha_{1,\b 1})du_1\wedge d\b u_1 \wedge dv_1}{|\td X|^4_g}+
\frac{(\alpha_2\alpha_{1,\b 2}-\alpha_1\alpha_{2,\b 2})du_1\wedge dv_1\wedge d\b v_1 }{|\td X|^4_g}.
\eeqs
Now we have the following:
\beqs
\alpha_1 &=& g_{1\b 1}\overline{u_1(a+v_1)}-g_{1\b 2}\b v_1^2,  \\
\alpha_2 &=& g_{2\b 1}\overline{u_1(a+v_1)}-g_{2\b 2}\b v_1^2,   \\
\alpha_{1,\b 1} &=& g_{1\b 1,\b 1}\overline{u_1(a+v_1)}-g_{1\b 2,\b 1}\b v_1^2+g_{1\b 1}\overline{(a+v_1)}, \\
\alpha_{1,\b 2} &=& g_{1\b 1,\b 2}\overline{u_1(a+v_1)}-g_{1\b 2,\b 2}\b v_1^2+g_{1\b 1}\b u_1-2g_{1\b 2}\b v_1,  \\
\alpha_{2,\b 1} &=& g_{2\b 1,\b 1}\overline{u_1(a+v_1)}-g_{2\b 2,\b 1}\b v_1^2+g_{2\b 1}\overline{(a+v_1)},  \\
\alpha_{2,\b 2} &=& g_{2\b 1,\b 2}\overline{u_1(a+v_1)}-g_{2\b 2,\b 2}\b v_1^2+g_{2\b 1}\b u_1-2g_{2\b 2}\b v_1.
\eeqs

\begin{proof}[Proof of Lemma \ref{lem:main}:]

To compute the limit
$$\lim_{r\to 0}\int_{\p B_r }\phi\eta\wedge\b\p\eta,$$
We use scaling: Set $(u_1,v_1)=(r^2u,rv)$, and for a function $f(u_1,v_1,\b u_1,\b v_1)$, the function $f^{(r)}$ is defined to be
$$f^{(r)}(u,v,\b u,\b v)=f(r^2u,rv,r^2\b u,r\b v).$$
Now in the coordinate $(u,v)$, the boundary $\partial B_r$ becomes $$S_r:=\{(u,v)|g^{(r)}_{1\bar 1}|u(a+rv)|^2-2 Re(g^{(r)}_{1\bar 2}u(a+rv)\b v^2)+g^{(r)}_{2\bar 2}|v|^4=1 \}.$$ Recall that on $\partial B_r$, we have $|\td X|_g^2\equiv r^4$. Then we have:
\beqs
\int_{\p B_r }\phi\eta\wedge\b\p\eta &=& \int_{S_r} \phi^{(r)}
\frac{r^5(\alpha^{(r)}_1\alpha^{(r)}_{2,\b 1}-\alpha^{(r)}_2\alpha^{(r)}_{1,\b 1})du\wedge d\b u \wedge dv+
r^4(\alpha^{(r)}_2\alpha^{(r)}_{1,\b 2}-\alpha^{(r)}_1\alpha^{(r)}_{2,\b 2})du\wedge dv\wedge d\b v }
{r^8}\\
&=& \int_{S_r} \phi^{(r)}
\frac{r(\alpha^{(r)}_1\alpha^{(r)}_{2,\b 1}-\alpha^{(r)}_2\alpha^{(r)}_{1,\b 1})du\wedge d\b u \wedge dv+
(\alpha^{(r)}_2\alpha^{(r)}_{1,\b 2}-\alpha^{(r)}_1\alpha^{(r)}_{2,\b 2})du\wedge dv\wedge d\b v }
{r^4}.
\eeqs
Now when $r\to 0$, for any function $f$ we have $f^{(r)}\to f(p_1)$. Moreover, we have that
$$\Big(  g^{(r)}_{1\bar 1}|u(a+rv)|^2-2Re(g^{(r)}_{1\bar 2}u(a+rv)\b v^2)+g^{(r)}_{2\bar 2}|v|^4  \Big)^2
\to Q_0(au,-v^2),$$
where $Q_0$ is the hermitian quadratic form defined by $g_{i\b j}(p_1)$. Note that $Q_0(au,-v^2)$ is invariant under the symmetries $(u,v)\mapsto (u,-v)$ and $(u,v)\mapsto (-u,\sqrt{-1}v)$. Direct computation shows that
\beqs
\frac{\alpha^{(r)}_1\alpha^{(r)}_{2,\b 1}-\alpha^{(r)}_2\alpha^{(r)}_{1,\b 1}}{r^3}&=& \frac{\overline{av^2}\det g^{(r)}}{r}+\b v^3 \det g^{(r)}+r(\dots)\\
\frac{\alpha^{(r)}_2\alpha^{(r)}_{1,\b 2}-\alpha^{(r)}_1\alpha^{(r)}_{2,\b 2}}{r^4}&=& \frac{2\overline{auv}\det g^{(r)}}{r}+
\overline{uv^2}(\dots)+\overline{u^2(a+rv)^2}(\dots)+\overline{uv^2(a+rv)}(\dots)\\
& &+\overline{v^4}(\dots)+r(\dots).
\eeqs

We claim that when taking limit, we need only to consider the terms with the factor $\frac{1}{r}$. First, for terms with a factor $r$, the limit vanishes automatically. For other terms without the factor $\frac{1}{r}$, the integration operation commutes with taking limit, and we can use the special symmetries of $Q_0(au,-v^2)$ to prove that the limit integral also vanishes. To sum up, we have
\begin{lem}We have
\beqs
  \lim_{r\to 0} \int_{\p B_r }\phi\eta\wedge\b\p\eta &=&\lim_{r\to 0}\frac{\b a}{r}\int_{S_r}
 \phi^{(r)} \det g^{(r)}\Big(\b v^2du\wedge d\b u \wedge dv+2\overline{uv}du\wedge dv\wedge d\b v \Big)\\
 &=& \lim_{r\to 0}\frac{\b a \Phi(r)}{r},
\eeqs
where $\Phi(r)$ is the integral over $S_r$.
\end{lem}

Now we use the Taylor expansion of the function $\phi^{(r)} \det g^{(r)}$, and using the symmetry of $S_0$, we have
\beqs
\lim_{r\to 0}\frac{\Phi(r)}{r}&=& (\phi\det g)(0)\lim_{r\to 0}\frac{1}{r}\int_{\td B_r}-4\b v du\wedge d\b u\wedge dv\wedge d\b v\\
& &+\frac{\partial}{\partial v_1}(\phi\det g)(0)\int_{S_0} |v|^2(\b v du\wedge d\b u\wedge dv +2\b u du\wedge dv\wedge d\b v),
\eeqs
where $\td B_r$ is the image of $B_r$ under the coordinate change. Next we evaluate the second integral. Since under the degree 2 map $(u,v)\mapsto (au,-v^2)$, the surface $S_0$ becomes
 $$\td S=\{(s,t)|Q_0(s,t)=1\}.$$
 So we have
 \beqs &&
 \int_{S_0} |v|^2(\b v du\wedge d\b u\wedge dv +2\b u du\wedge dv\wedge d\b v)\\&=&
 2\frac{1}{2|a|^2} \int_{\td S}\b t ds\wedge d\b s\wedge dt +\b s ds\wedge dt\wedge d\b t\\
 &=& -\frac{2}{|a|^2}\int_{Q_0(s,t)\leq 1} ds\wedge d\b s \wedge dt\wedge d\b t\\
 &=& \frac{2}{|a|^2\det g(0)}\int_{|s|^2+|t|^2\leq 1} (\sqrt{-1})^2 ds\wedge d\b s \wedge dt\wedge d\b t\\
 &=& \frac{4\pi^2}{|a|^2\det g(0)}.
 \eeqs
For the first limit we have the following lemma:

\begin{lem}\label{lem:appendix}
We have
  $$\lim_{r\to 0}\frac{1}{r}\int_{\td B_r}\b v du\wedge d\b u\wedge dv\wedge d\b v =
  \pi^2\Big(\frac{\pd {}{v_1}(\det g)(0)}{|a|^2(\det g (0))^2}+\frac{\b a}{|a|^4\det g(0)}\Big)$$
\end{lem}

$\ $\newline

Now combining the above results, we get Lemma \ref{lem:main}.
\end{proof}

\begin{proof}[Proof of Lemma \ref{lem:appendix}:]
  To compute the integral
$$\int_{\td B_r}\b v du\wedge d\b u\wedge dv\wedge d\b v ,$$
where $\td B_r$ is given by
$$\td B_r=\Big\{(u,v)\,\Big|\, g_{1\b 1}^{(r)}|u(a+rv)|^2-g^{(r)}_{1\b 2}u(a+rv)\b v^2-g^{(r)}_{2\b 1}\overline{u(a+rv)} v^2+g^{(r)}_{2\b 2}|v|^4\leq 1\Big\}.$$
Consider the following differentiable coordinate transformation:
$$
s= \sqrt{g_{1\b 1}^{(r)}}u(a+rv)-\frac{g^{(r)}_{2\b 1}}{\sqrt{g_{1\b
1}^{(r)}}}v^2,\quad
t= \Big(\frac{\det g^{(r)}}{g_{1\b 1}^{(r)}}\Big)^{\frac{1}{4}}v.
$$
Then the domain $\td B_r$ is transformed to $\Om_0=\{(s,t)||s|^2+|t|^4\leq 1\}$. Direct computation shows that
$$\pd{t}{u}=O(r^2),\quad \pd{t}{\b u}=O(r^2).$$
So we have
$$dt\wedge d\b t=\Big(\Big|\pd{t}{v}\Big|^2-\Big|\pd{t}{\b v}\Big|^2\Big) dv\wedge d\b v+O(r^2).$$
Note that $\pd{t}{\b v}=O(r)$,   we have
$$dt\wedge d\b t=\Big|\pd{t}{v}\Big|^2dv\wedge d\b v+O(r^2).$$
It follows that
\beqs ds\wedge d\b s\wedge dt\wedge d\b t&=&\Big(\Big|\pd{t}{v}\Big|^2\Big(\Big|\pd{s}{u}\Big|^2-\Big|\pd{s}{\b u}\Big|^2\Big)+O(r^2)\Big)du\wedge d\b u\wedge dv\wedge d\b v\\
&=& \Big(\Big|\pd{t}{v}\Big|^2\Big|\pd{s}{u}\Big|^2+O(r^2)\Big)du\wedge d\b u\wedge dv\wedge d\b
v,
\eeqs
where we used $\pd{s}{\b u}=O(r^2)$ in the last inequality.\\

Now we compute $\Big|\pd{t}{v}\Big|^2\Big|\pd{s}{u}\Big|^2$. We have
\beqs
\pd{s}{u}=(a+rv)\sqrt{g^{(r)}_{1\b 1}}+O(r^2),
\eeqs
So
\beqs
\Big|\pd{s}{u}\Big|^2&=& (|a|^2+rv\b a+r\b v a)\Big(g_{1\b 1}(0)+rv\pd{g_{1\b 1}}{v_1}(0)+r\b v\pd{g_{1\b 1}}{\b v_1}(0)\Big)+O(r^2)\\
&=& |a|^2g_{1\b 1}(0)+rv\Big(|a|^2\pd{g_{1\b 1}}{v_1}(0)+\b a g_{1\b 1}(0)\Big)
+r\b v\Big(|a|^2\pd{g_{1\b 1}}{\b v_1}(0)+ a g_{1\b 1}(0)\Big)+O(r^2).
\eeqs
Similarly, we have
\beqs
\Big|\pd{t}{v}\Big|^2&=& \Big(\frac{\det g(0)}{g_{1\b 1}(0)}\Big)^{-\frac{3}{2}}
\Big(
\Big(\frac{\det g(0)}{g_{1\b 1}(0)}\Big)^2+\frac{3}{4}rv \frac{\det g(0)}{g_{1\b 1}(0)}\pd{}{v_1}\Big(\frac{\det g}{g_{1\b 1}}\Big)(0)\\
& & +\frac{3}{4}r\b v \frac{\det g(0)}{g_{1\b 1}(0)}\pd{}{\b v_1}\Big(\frac{\det g}{g_{1\b 1}}\Big)(0)\Big)+O(r^2).
\eeqs
So we get
\beq
\Big|\pd{t}{v}\Big|^2\Big|\pd{s}{u}\Big|^2=
\Big(\frac{\det g(0)}{g_{1\b 1}(0)}\Big)^{-\frac{3}{2}}(C_0+rvC_1+r\b v C_2)+O(r^2),
\label{eqA1}\eeq
where
\beq
C_0=|a|^2\frac{(\det g(0))^2}{g_{1\b 1}(0)}
\label{eqA2}\eeq
and
\beq
C_1=\frac{\frac{3}{4}|a|^2g_{1\b 1}(0)\det g(0)\pd{\det g}{v_1}(0)+(\det g(0))^2(\frac{1}{4}|a|^2\pd{g_{1\b 1}}{v_1}(0)+\b a g_{1\b 1}(0))}{(g_{1\b 1}(0))^2}.
\label{eqA3}\eeq
Using (\ref{eqA1}), we have
\beqs &&
\lim_{r\to 0}\frac{1}{r}\int_{\td B_r}\b v du\wedge d\b u\wedge dv\wedge d\b v \\&=&
\lim_{r\to 0}\frac{1}{r}\Big(\frac{\det g(0)}{g_{1\b 1}(0)}\Big)^{\frac{3}{2}}
\int_{\Om_0}\Big(\frac{\det g}{g_{1\b 1}}\Big)^{-\frac{1}{4}}
\frac{\b t ds\wedge d\b s\wedge dt\wedge d\b t}{C_0+rvC_1+r\b v C_2+O(r^2)}\\
&=&\Big(\frac{\det g(0)}{g_{1\b 1}(0)}\Big)^{\frac{3}{2}}\lim_{r\to 0}
\frac{1}{r}\int_{\Om_0}
\Big(\Big(\frac{\det g}{g_{1\b 1}}(0)\Big)^{-\frac{1}{4}}-\frac{rt}{4}\Big(\frac{\det g}{g_{1\b 1}}(0)\Big)^{-\frac{3}{2}}\pd{}{v_1}\Big(\frac{\det g}{g_{1\b 1}}\Big)(0)\Big)\\
& &\cdot\frac{1}{C_0}\Big(1-rt\frac{C_1}{C_0}\Big(\frac{\det g}{g_{1\b 1}}(0)\Big)^{-\frac{1}{4}}\Big)\;\b t ds\wedge d\b s\wedge dt\wedge d\b t\\
&=&\Big(\frac{\det g(0)}{g_{1\b 1}(0)}\Big)^{\frac{3}{2}}
\Big(\frac{1}{4C_0}\Big(\frac{\det g}{g_{1\b 1}}(0)\Big)^{-\frac{3}{2}}\pd{}{v_1}\Big(\frac{\det g}{g_{1\b 1}}\Big)(0)+\frac{C_1}{C_0^2}\Big(\frac{\det g(0)}{g_{1\b
1}(0)}\Big)^{-\frac{1}{2}}\Big)
\int_{\Om_0}(-1)|t|^2ds\wedge d\b s\wedge dt\wedge d\b t.
\eeqs
Now using the 2-1 mapping $(s,t)\mapsto (s,t^2)$, we have
\beqs
\int_{\Om_0}(-1)|t|^2ds\wedge d\b s\wedge dt\wedge d\b t=
\frac{1}{2}\int_{|s|^2+|t|^2\leq 1} (-1)ds\wedge d\b s\wedge dt\wedge d\b t= \pi^2.
\eeqs
Finally using (\ref{eqA2}) and (\ref{eqA3}), we get
$$\lim_{r\to 0}\frac{1}{r}\int_{\td B_r}\b v du\wedge d\b u\wedge dv\wedge d\b v =
  \pi^2\Big(\frac{\pd {}{v_1}(\det g)(0)}{|a|^2(\det g (0))^2}+\frac{\b a}{|a|^4\det
  g(0)}\Big).$$\\

\end{proof}

\begin{rem}We can also prove Lemma \ref{lem:main} by a direct method without using
Poincar\'e's theorem, but the calculation is much more complicated.

\end{rem}

\section{Blowing up at non-isolated zeroes}\label{sec4}
\subsection{The zero set of holomorphic vector fields}
In this section, we consider the non-isolated case. Let $Z_{\la}$ be
a one dimensional component of the zero set of $X$ on a K\"ahler
surface $M$. We choose a coordinate $(z, w)$ on a neighborhood $U$
of $p$ such that $Z_{\la}\cap U=\{z=0\}.$ Therefore, $X$ can be
locally written as
$$X=z\cdot h(z, w)\pd {}z+z\cdot k(z, w)\pd {}{w},$$
where $h(z, w)$ and $k(z, w)$ are holomorphic functions on $U.$
Since $X$ is non-degenerate at $Z_{\la},$ we have
$$h(0, w)\neq 0,\quad (0, w)\in U$$
and we can assume that \beqs h(z, w)&=& a_0+a_1z+a_2w+\sum_{i+j\geq
2}a_{ij}z^iw^j,\quad a_0\neq 0,\\
k(z, w)&=&b_0+b_1z+b_2w+\sum_{i+j\geq 2}b_{ij}z^iw^j. \eeqs

 Let
$\pi: \td M\ri M$ be the blowing up of $M$ at the point $p$. We denote by $L$ the strict transform
of $Z_\la$ under $\pi$, and by $Z$ the zero locus of $\td X$ over $\td M$. Then obviously $L\subset Z$. Now we study the zeroes of $\td X$ on the exceptional divisor $E$. Choose coordinate charts $\td U_1$ and $\td U_2$ of $\td U=\pi^{-1}(U)$
as in Section 3: \beqs \td U_1&=&\{((z, w), [\zeta,
\eta])|\,z\eta=w\zeta,\;\zeta\neq 0\}\subset U\times \CC\PP^1,\\
 \td U_2&=&\{((z, w), [\zeta,
\eta])|\,z\eta=w\zeta,\;\eta\neq 0\}\subset U\times \CC\PP^1.\eeqs
We choose coordinates $u_1=z, v_1=\frac {\eta}{\zeta}$ on $\td U_1$
and we have $E\cap \td U_1=\{u_1=0\} $ and $L\cap \td
U_1=\emptyset.$
 On $\td U_1$ $\td X:=\pi^*X$ can be written as
$$\td X=\td X^1(u_1, v_1)\pd {}{u_1}+\td X^2(u_1, v_1)\pd {}{v_1},$$
where $\td X^1$ and $\td X^2$ are given by \beqs \td
X^1&=&a_0u_1+a_1u_1^2+a_2u_1^2v_1+\sum_{i+j\geq 2}a_{ij}u_1^{i+j+1}v_1^j,\\
\td
X^2&=&b_0+b_1u_1-a_0v_1+(b_2-a_1)u_1v_1-a_2u_1v_1^2+\sum_{i+j\geq
2}\,(b_{ij}v_1^j-a_{ij}v_1^{j+1})u_1^{i+j}. \eeqs The zero set of
$\td X$ on $\td U_1$ is given by
 $$Z\cap \td U_1=\{p_1\},\quad p_1:(u_1, v_1)=(0, \frac {b_0}{a_0})$$
which is a non-degenerate zero of $\td X.$

On the other hand, we
choose coordinates $u_2=\frac {\zeta}{\eta}, v_2=w$ on $\td U_2$ and
we have $E\cap \td U_2=\{v_2=0\}$ and $L\cap \td U_2=\{u_2=0\}.$
Note that
 $\td
X$ can be written as
$$\td X=\hat X^1(u_2, v_2)\pd {}{u_2}+\hat X^2(u_2, v_2)\pd{}{v_2},$$
where \beqs \hat X^1(u_2,
v_2)&=&a_0u_2-b_0u_2^2+a_2u_2v_2+(a_1-b_2)u_2^2v_2-b_1u_2^3v_2+\sum_{i+j\geq
2}(a_{ij}-b_{ij}u_2)u_2^{i+1}v_2^{i+j},\\\hat X^2(u_2,
v_2)&=&u_2v_2\Big(b_0+b_1u_2v_2+b_2v_2+\sum_{i+j\geq
2}\,u_2^iv_2^{i+j}\Big).
 \eeqs Therefore, the zero set of $\td X$ on $\td U_2$
 consists of the
 following cases:
 \begin{itemize}
   \item If $b_0= 0$, then $Z\cap \td U_2=L\cap \td U_2;$
   \item If $b_0\neq 0,$ then
   $Z\cap \td U_2=(L\cap \td U_2)\cup\{q_1\}$, where $q_1:(u_2, v_2)=( \frac {a_0}{b_0},
   0).$ One can check easily that $q_1=p_1\in \td U_1$.\\
 \end{itemize}

Combining the above results, we have
\begin{lem}\label{lemzero2}The zero set $Z$ of $\td X$ on $\td U$ is given by
$Z\cap \td U=(L\cap \td U)\cup\{p_1\}$, where $p_1\in E$ is the point $((0,0),[a_0,b_0])\in U\times \CC\PP^1$.

\end{lem}

\subsection{The local Futaki invariant}

In this section we will calculate the Futaki invariant of the
blow-up $\pi: \td M\ri M$ of K\"ahler surface $M$ at a point $p\in
l$ where $l$ is a one-dimensional component of the set of $X$.
  We assume that $X$ is
non-degenerate on $l$. As before  $X$ can be naturally extended to a
holomorphic vector field $\td X$ on $\td M$. We would like to
compute the Futaki invariant of $(\td \Om_{\ee}, \td X)$ on $\td M$
where $\td
\Om_{\ee}=\pi^*\Om-\ee c_1( [E])$, where $E=\pi^{-1}(p)$ is the exceptional divisor. \\

Let $\zero(X)=\cup_{\la\in \La}Z_{\la}$ be the zero set of $X$ on
$M$ and $Z_0=l$ where $l$ is the curve containing the point $p$ as
above. Let $L=\pi^*l-E$ be the strict transform of $l$. Then $\td X$
vanishes on $L$ and the zero set $\td Z$ of $\td X$ on $\td M$ can be divided into three types according to Lemma \ref{lemzero2}: $\td Z=L\cup\{q\}\cup_{\la\in \La,\la\neq 0}\td Z_\la$, where $\td Z_\la$ is the strict transform of $Z_\la$, and $q\in E$ is an isolated zero point of $\td X$ which does not lie on $L$.
Let $\td \mu=\frac {c_1(\td M)\cdot \td \Om_{\ee}}{\td \Om_{\ee}^2}$
and we define
$$\dd:=\td \mu-\mu=-\frac 1{\Om^2}\ee+O(\ee^2).$$
With these notations, we have

\begin{lem}\label{lemblowup2}
$$f_{\td M}(\td \Om_{\ee}, \td X)=f_{M}(  \Om ,   X)-\frac 2{3}\dd J_{M}(  \Om ,   X)
+f_{q}(\td \Om_{\ee}, \td X)-\Big(2B_0-2
B_0^2A_0^{-1}(\mu+\dd)\Big) \ee-\frac 23A_0^{-2}B_0^3(\mu+\dd).$$

\end{lem}
\begin{proof}The Futaki invariant of $(\td \Om_{\ee}, \td X)$ on $\td
M$ is given by
$$f_{\td M}(\td \Om_{\ee}, \td X)=\sum_{\la\in \La,\la\neq 0}f_{\td Z_{\la}}(\td \Om_{\ee}, \td X)
+f_L(\td \Om_{\ee}, \td X)+f_{q}(\td
\Om_{\ee}, \td X).$$
 Note that for any $\la\in \La, \la\neq 0$, we have \beqs
f_{\td Z_{\la}}(\td \Om_{\ee}, \td X)&=&I_{\td Z_{\la}}(\td
\Om_{\ee}, \td X)-\frac n{n+1}\td \mu
J_{\td Z_{\la}}(\td \Om_{\ee}, \td X)\\
&=&I_{  Z_{\la}}(  \Om ,   X)-\frac n{n+1}(\mu+\dd) J_{ Z_{\la}}(
\Om_{\ee},   X)\\
&=&f_{  Z_{\la}}(  \Om ,   X)-\frac n{n+1}\dd J_{ Z_{\la}}(
\Om_{\ee}, X),\eeqs where we used the fact that $I_{\td Z_{\la}}(\td
\Om_{\ee}, \td X)=I_{  Z_{\la}}(  \Om ,   X)$ and $J_{\td
Z_{\la}}(\td \Om_{\ee}, \td X)=J_{  Z_{\la}}(  \Om ,   X)$ since
$\td Z_{\la} $ and $E$ are disjoint for $\la\neq 0$.  Note that
$$\td \Om_{\ee}([L])=\Om([l])-\ee,\quad c_1(\td M)([L])=c_1(M)([l])-1,\quad g(L)=g(l).$$
Thus, using Corollary \ref{corfutaki} we have \beqn f_L(\td \Om_\ee, \td
X)&=& (2B_{0}-2\td \mu B_{0}^2A_{0}^{-1})\td \Om_{\ee}([L]) +(\frac
{2\td \mu}{3}A_{0}^{-2}B_{0}^3)c_1(\td
M)([L])+(A_{0}^{-1}B_{0}^2-\frac {2\td
\mu}{3}A_{0}^{-2}B_{0}^3)(2-2g(L))\nonumber\\
&=&f_l(\Om, X)-\frac 23\dd J_l(\Om, X)-\Big(2B_0-2
B_0^2A_0^{-1}(\mu+\dd)\Big) \ee-\frac 23A_0^{-2}B_0^3(\mu+\dd).
  \eeqn\\
Combine these formulas, we proved the lemma.\\\end{proof}

\begin{theo}\label{theo:main2}Let $\pi: \td M\ri M$ be the blow-up of $M$ at $p\in l$ where $l$ is
a one-dimensional component of the zero set of $X$. If $X$ is
non-degenerate on $M$, then we have
$$f_{\td M}(\td \Om_{\ee}, \td X)=f_M(\Om,   X)+\nu(\Om, X)\cdot \ee+O(\ee^2),$$
where $\nu(\Om, X)$ is given by
$$\nu(\Om, X)=-2\tr_{\Om}(X)(p)+\frac
{2}{3\Om^2}J_{M}(\Om, X).$$

\end{theo}

\begin{proof}By Lemma \ref{lemzero2}, the zero set of $\td X$ on $\td U$ is given
by $(L\cap \td U)\cup\{p_1\}$ where $p_1:(u_1, v_1)=(0, 0).$ Note
that using the local expression of $\td X$ at $p_1$ we have
$$\td A_{p_1}=0,\quad \td B_{ p_1}=\theta_X(p)-a_0\ee,\quad \td C_{ p_1}=-a_0^2. $$
Thus, we have
$$f_{ p_1}(\td \Om_{\ee}, \td X)=\frac {2(\mu+\dd)(\te_X(p)-a_0\ee)^3}{3a_0^2}.$$
Combining this with Lemma \ref{lemblowup2}, we have \beqs f_{\td
M}(\td \Om_{\ee}, \td X)&=&f_{M}(  \Om ,   X)-\frac 2{3}\dd J_{M}(
\Om
, X)+\frac {2(\mu+\dd)(\te_X(p)-a_0\ee)^3}{3a_0^2}\\
&&-\Big(2\te_X(p)-\frac {2\mu\te_X^2(p)}{a_0}\Big)\ee-\frac
{2\te_X^3(p)\mu}{3a_0^2}-\frac {2\te_X^3(p)\dd}{3a_0^2}+O(\ee^2)\\
&=&f_{M}(  \Om ,   X)-\frac 2{3}\dd J_{M}( \Om ,
X)-2\theta_X(p)\ee+O(\ee^2). \eeqs The theorem is proved.

\end{proof}

\section{Examples}\label{secEg}

In this section, we apply our theorem to some examples. Actually, we can get very explicit formulas as we do in the proof our theorems. For simplicity, we only write down the first order expansion, which in many cases suffices to prove the non-existence of cscK metrics.

\subsection{$\CC\PP^1\times \CC\PP^1$ blowing up two points}

 Let $M=\CC\PP^1\times\CC\PP^1$, and $p_1=(0,0),p_2=(\infty,\infty)$ are two points of $M$. We blow up $p_1,p_2$ to get $\td M_1$ and denote the blowing up map by $\pi$, with
exceptional divisors $E_1$ and $E_2$. $\td M_1$ can also be realized as $\CC\PP^2$ blowing up three generic points, and is a toric Fano manifold. Write the homogeneous
coordinates of the two factors of $\CC\PP^1\times\CC\PP^1$ by $[z_0,
z_1]$ and $[w_0, w_1]$.Let $z=\frac {z_1}{z_0}=\frac 1{\td z}$ and $w=\frac
{w_1}{w_0}=\frac 1{ \td w}.$  Then one can check easily that the
vector fields $z\pd{}{z}$ and $w\pd{}{w}$
extend naturally to holomorphic vector fields on $\td M_1$ and they form a
 basis of $\frak h_0(\td M_1)$. We denote these two vector fields by  $Z$ and $W$
 respectively.

 Any K\"ahler class on $M$ has the form $\Om_{a,b}=ac_1([H_1])+bc_1([H_2])$, $a,b>0$, where $H_1, H_2$ are divisors in
$M$ defined by $z=const$ and $w=const$, respectively. Since any K\"ahler class on $M$ admits a cscK metric, we know that the Futaki invariant vanishes identically on $\frak h(M)$. Now we consider the K\"ahler class on $\td M_1$:
$$\td \Om_{a,b,\ee_1,\ee_2}=\pi^*\Om_{a,b}-\ee_1c_1([E_1])-\ee_2c_1([E_2]).$$
We want to compute $f_{\td M_1}(\td \Om_{a,b,\ee_1,\ee_2},Z):=f_Z(a,b,\ee_1,\ee_2)$. (The computation of $f_W$ is the same.)

We choose the following K\"ahler form in the class $\Om_{a,b}$ on $M$:
$$\omega=a\Big(\partial\bar\partial\log(|z_0|^2+|z_1|^2)\Big)+b\Big(\partial\bar\partial\log(|w_0|^2+|w_1|^2)\Big),$$
then we can choose the holomorphy potential of $Z$ to be:
$$\theta_Z=-a\frac{|z_1|^2}{|z_0|^2+|z_1|^2}.$$
By symmetry, we have $\underline{\te_Z}=-\frac{a}{2}$. We have the following:
$$\nu_{p_1}=-2(\te_Z-\underline{\te_Z})(p_1)=-2(0-(-\frac{a}{2}))=-a,\quad \nu_{p_2}=-2(\te_Z-\underline{\te_Z})(p_2)=-2(-a-(-\frac{a}{2}))=a.$$
By Theorem \ref{theo:main2}, we have
$f_Z(a,b,\ee_1,\ee_2)=a(\ee_2-\ee_1)+o(|\ee|).$

\begin{cor}
  Let $\td M_1$ and $\td \Om_{a,b,\ee_1,\ee_2}$ be as above. Then for small $\ee_i>0$ with $\ee_1\neq\ee_2$, there are no cscK metrics in the class $\td \Om_{a,b,\ee_1,\ee_2}$.
\end{cor}

\begin{rem}
By our method of proving the main theorems, one can actually have the following exact formula:
$$f_Z(a,b,\ee_1,\ee_2)=-2a(a+b-\ee_2)-\frac{2(2a+2b-\ee_1-\ee_2)}{3(2ab-\ee_1^2-\ee_2^2)}(\ee_1^3-\ee_2^3+3a\ee_2^2-3a^2b).$$
\end{rem}

\subsection{$\CC\PP^1\times \CC\PP^1$ blowing up three points}

Let $M=\CC\PP^1\times\CC\PP^1$ and $\Om_{a,b}$ be as above, and $p_1=(0,0),p_2=(\infty,\infty), p_3=(0,\infty)$ are three points of $M$. We blow up $p_1,p_2,p_3$ to get $\td M_2$ and denote the blowing up map by $\pi$, with
exceptional divisors $E_1, E_2$ and $E_3$. $\td M$ is also a toric manifold. Write the homogeneous
coordinates of the two factors of $\CC\PP^1\times\CC\PP^1$ by $[z_0,
z_1]$ and $[w_0, w_1]$.Let $z=\frac {z_1}{z_0}=\frac 1{\td z}$ and $w=\frac
{w_1}{w_0}=\frac 1{ \td w}.$  Then the natural extensions of the
vector fields $z\pd{}{z}$ and $w\pd{}{w}$
form a basis of $\frak h_0(\td M_2)$. We denote these two vector fields by  $Z$ and $W$
 respectively.

The K\"ahler class we choose on $\td M_2$ is:
$$\td \Om_{a,b,\ee_1,\ee_2,\ee_3}=\pi^*\Om_{a,b}-\ee_1c_1([E_1])-\ee_2c_1([E_2])-\ee_3c_1([E_3]).$$
We want to compute $f_{\td M_2}(\td \Om_{a,b,\ee_1,\ee_2,\ee_3},Z):=f_Z(a,b,\ee_1,\ee_2,\ee_3)$ and similarly $f_W(a,b,\ee_1,\ee_2,\ee_3)$.

The holomorphy potential of $Z$ is still
$$\theta_Z=-a\frac{|z_1|^2}{|z_0|^2+|z_1|^2}.$$
And we still have $\underline{\te_Z}=-\frac{a}{2}$. We have the following:
$$\nu_{p_1}=-a, \nu_{p_2}=a, \nu_{p_3}=-a.$$
By Theorem \ref{theo:main2}, we have
$$f_Z(a,b,\ee_1,\ee_2,\ee_3)=a(\ee_2-\ee_1-\ee_3)+o(|\ee|).$$ Similarly, we have $$f_W(a,b,\ee_1,\ee_2,\ee_3)=b(\ee_2-\ee_1+\ee_3)+o(|\ee|).$$ Since $\ee_2-\ee_1-\ee_3$ and $\ee_2-\ee_1+\ee_3$ can not be both zero when all the $\ee_i$'s are positive, we have:

\begin{cor}
  Let $\td M_2$ and $\td \Om_{a,b,\ee_1,\ee_2,\ee_3}$ be as above. Then for $\ee_i>0$ small enough, there are no cscK metrics in the class $\td \Om_{a,b,\ee_1,\ee_2,\ee_3}$.
\end{cor}



\section{Relation with Stoppa's result}\label{secStoppa}

In this section, we point out the relation of Stoppa's theorem with
ours when the K\"ahler manifold is a polarized algebraic surface $(M,L)$, with K\"ahler class $\Om=c_1(L)$.
Let's first recall Stoppa's result.

Let $Z=\sum_i a_i p_i$ be a 0-dimensional cycle on a $n$-dimensional
 polarized algebraic manifold $(M,L)$,
  where $p_i\in M$ are different points and $a_i\in \mathbb{Z}_+$.
  We write $\tilde M:=Bl_ZM$ and denote by $p:\tilde M\to M$ the blowing up map,
   with exceptional divisor $E$.\footnote{As a manifold, $\tilde M$ is the same as $M$ blowing up all the points $p_i$. If we denote the exceptional divisors by $E_i$, then $E=\sum_i a_i E_i$.}
   Assume $X$ is a holomorphic vector field on $M$ that generates a holomorphic $\mathbb{C}^*$ action
  $\alpha(t)$.

Define $Z_t:=\alpha(t)Z$ and taking the flat closure of $\cup_{t\in
\mathbb{C}^*}Z_t\times \{t\}\subset M\times \mathbb{C}$, we get a
subscheme $Y$ of $M\times \mathbb{C}$. Then blowing up $Y$, we get
$\mathcal{X}=Bl_Y(M\times \mathbb{C})$, which is a test
configuration for $(\tilde M, \tilde L)$, where $\tilde L=\gamma p^*
L-E$ for some positive large integer $\gamma$, here $\frac{a_i}{\gamma}$ plays the same role as $\ee_i$ in our setting.
Stoppa got the
following formula for the algebraic Donaldson-Futaki invariant of
this test configuration: \beqn\label{eqStoppa1}
 F(\mathcal{X})=F(M,L,X)\gamma^n-\mathcal{CH}(\sum_i a_i^{n-1}p_i,\alpha)\frac{\gamma}{2(n-2)!}+O(1),
\eeqn where $F(M,L,X)$ is the algebraic Donaldson-Futaki invariant
of the product test configuration of $(M,L)$ with $\mathbb{C}^*$
action induced by $\alpha$, and $\mathcal{CH}(\sum_i
a_i^{n-1}p_i,\alpha)$ is the Chow weight of $\alpha(t)$ acting on
0-dimensional cycles.

In our case, the blowing up centers $p_i$ are non-degenerate  zero points of
$X$, so they are fixed by $\alpha(t)$, and hence
$Y=Z\times\mathbb{C}$. So $\mathcal{X}=Bl_ZM\times \mathbb{C}$ is a
product test configuration. In this case, both $F(\mathcal{X})$ and
$F(M,L,X)$ are the classical Calabi-Futaki invariants, up to a
universal constant factor (see \cite{[Do1]}). Observe that in this case, the holomorphic vector field
$\td X$ also have non-degenerate zero locus. This is because $X$ generates a $\CC^*$ action, so the linearization of
$X$ at any of its zero point is semisimple, and our Lemma \ref{lemzero} guarantees the non-degeneracy.

Now we give a formula for the Chow weight in this case, using the
potential of $X$. Assume $L^{\gamma}$ is very ample. For simplicity,
we also assume that the induced action of $\alpha$ on $H^0(X, \gamma
L)$ gives a 1-ps of $SL(N+1)$. We also assume that $(M,L)$ is
asymptotically Chow polystable.  First by Stoppa's work
(\cite{[Stoppa]} 14-15), we know that $\mathcal{CH}(\sum_i
a_ip_i,\alpha)=-\sum_i a_i^{n-1}\lambda(p_i)$. The definition of
$\lambda(p_i)$ is as follows. Since the $\mathbb{C}^*$ action
$\alpha(t)$ preserves the fiber of $L$ over $p_i$, we have a
well-defined notion of weight for this action. This is
$\lambda(p_i)$. We also write the induced linear
 $\mathbb{C}^*$ action on $\mathbb{P}^N$ as $\alpha(t)$. Suppose the image of $p_i$ is the point
$[1,0,\dots,0]$, and the action of $\alpha(t)$ is in a diagonal form
$diag(t^{\lambda_0},\dots,t^{\lambda_N})$. Then
$\lambda_0=-\gamma\lambda(p_i)$.

The holomorphy potential of $X$ is defined by the equation $-\b\p
\theta_X=i_X\omega$. Applying the $d$ operator, we get
$-\p\b\p\theta_X=L_X\omega.$ We can choose a special metric to
compute $\theta_X$.
 So let's assume that
$\omega$ is the pull-back metric
$$\frac{1}{\gamma}\p\b\p\log (|Z_0|^2+\dots+|Z_n|^2).$$
Since we assume that $(M,L)$ is asymptotically Chow polystable, we
can choose $\omega$ to be a balanced metric. The real 1-parameter
group associated with $X$ (and $\alpha(t)$) is
$\beta(s)=diag(e^{\lambda_0 s},\dots,e^{\lambda_N s})$. Then by the
definition of Lie derivatives, we have
$$L_X\omega =L_{ReX}\omega=\frac{d}{ds}|_{s=0}\beta(s)^*\omega=\frac{1}{\gamma}\p\b\p\frac{2\lambda_0|Z_0|^2+\dots+2\lambda_N |Z_N|^2}{|Z_0|^2+\dots+|Z_N|^2}.$$
The first equality is because the action is Hamiltonian. So we can
take
$$\theta_X=-\frac{1}{\gamma}\frac{2\lambda_0|Z_0|^2+\dots+2\lambda_N
|Z_N|^2}{|Z_0|^2+\dots+|Z_N|^2}$$ where the right handside means
restriction to the image of $M$ under Kodaira's embedding map.
Evaluate at $p_i$, we get
$\theta_X(p_i)=-\frac{2}{\gamma}\lambda_0=2\lambda(p_i)$. Since $\omega$ is balanced, we have $\underline{\te_X}=0$.
So $\nu_{p_i}(\Om,X)=-4\lambda(p_i)$. So in this case, our result coincides with (\ref{eqStoppa1}), up to a universal constant factor.\\

\noindent Haozhao Li,\\
 Department of Mathematics,  University of Science and Technology of China, Hefei, 230026, Anhui
province, China
\\ and \\
Wu Wen-Tsun Key Laboratory of Mathematics, USTC, Chinese Academy of Sciences, Hefei
230026, Anhui, PR China\\
Email: hzli@ustc.edu.cn\\

\noindent Yalong Shi,\\
Department of Mathematics, Nanjing University, Nanjing, 210093,
Jiangsu province, China. Email: shiyl@nju.edu.cn.

\end{document}